\documentclass[11pt,reqno]{amsart}
\usepackage{amssymb,amscd,amsbsy}
\usepackage{amssymb,amscd,amsbsy,mathrsfs}
\newcommand{\lb}{\linebreak}

\newcommand{\s}{\sigma}

\newcommand{\f}{\varphi}

\newcommand{\D}{\Delta}
\renewcommand{\L}{\Lambda}

\newcommand{\F}{{\mathscr F}}

\newcommand{\X}{{\mathscr X}}
\newcommand{\Y}{{\mathscr Y}}

\newcommand{\R}{{\Bbb R}}
\newcommand{\Z}{{\Bbb Z}}

\newcommand{\0}{{\boldsymbol{0}}}

\newcommand{\bs}{\boldsymbol}

\newcommand{\bS}{{\boldsymbol S}}

\newcommand{\rf}[1]{(\ref{#1})}

\newcommand{\df}{\stackrel{\mathrm{def}}{=}}

\newcommand{\re}{\operatorname{Re}}

\newcommand{\supp}{\operatorname{supp}}

\newcommand{\const}{\operatorname{const}}

\newcommand{\eeq}{\end{equation}}
\newcommand{\beq}{\begin{equation}}
\newcommand{\bay}{\begin{eqnarray}}
\newcommand{\ba}{\begin{align*}}
\newcommand{\ea}{\end{align*}}
\newcommand{\ey}{\end{eqnarray}}
\newcommand{\bey}{\begin{eqnarray*}}
\newcommand{\eey}{\end{eqnarray*}}

\newcommand{\be}{\infty}

\newcommand{\bl}{\blacksquare}

\newcommand{\Pf}{{\bf Proof. }}
\newcommand{\im}{\operatorname{Im}}
\renewcommand{\re}{\operatorname{Re}}

\newtheorem{thm}{\hspace{\parindent}Theorem}[section]

\newtheorem{cor}[thm]{\hspace{\parindent}Corollary}
\newtheorem{lem}[thm]{\hspace{\parindent}Lemma}

\theoremstyle{remark}

\newtheorem*{rem*}{Remark}

\newcommand{\Dom}{{\rm Dom}}

\newcommand\fM{\frak M}

\newcommand{\rd}{{\rm d}}

\newcommand\ri{{\rm i}}

\begin{document}

\title{Functions of pairs of commuting self-adjoint operators under relatively bounded perturbations}
\author{A.B. Aleksandrov and V.V. Peller}

\maketitle

\begin{abstract}
We study the behaviour of functions of pairs of commuting self-adjoint operators under perturbations by relatively bounded operators. We obtain analogs of our earlier results for functions of a single self-adjoint operator under relatively bounded perturbations. The main tool is double operator integrals
\end{abstract}

\numberwithin{equation}{section}

\setcounter{section}{0}
\section{\bf Introduction}
\label{In}
\setcounter{equation}{0}

\

This paper can be considered as a sequel of our earlier papers \cite{Pe}, \cite{APPS} and \cite{AP3}. We study the behaviour of functions of normal operators (or, functions of pairs of commuting self-adjoint operators under perturbations.

Let $N$ and $M$ be (not necessarily bounded) normal operators with the same domain. We are going to consider $N$ as an initial operator and $M$ is a perturbed operator. The operator $M-N$ will be interpreted as the perturbation. Recall that $M-N$ is said to be a {\it relatively bounded perturbation of} $N$ if $\Dom(M)\subset\Dom(N)$ and
the following inequality hods
$$
\|(M-N)x\|\le\const\big(\|Nx\|+\|x\|),\quad x\in\Dom(M),
$$
where $\Dom(N)$ stands for the domain of $N$.

It turns out that it is more natural to consider the assumption that $\re M-\re N$ is a relatively bounded perturbation of $\re N$
and $\im M-\im N$ is a relatively bounded perturbation of $\im N$. It is easy to see that this assumption implies that $M-N$ is a relatively bounded perturbation of $N$. The converse, however, is false. Indeed, let $A$ be an unbounded self-adjoint operator. Put
$$
N=\ri A\quad\mbox{and}\quad M=A+\ri A.
$$
Clearly, $M-N$ is a relatively bounded perturbation of $N$. On the other hand, $\re M-\re N=A$ and $\re M-\re N=A$ and, clearly, 
$\re M-\re N$ is not a relatively bounded perturbation of $\re N=\0$.

Thus, we are going to consider the case when both $\re M-\re N$ is a relatively bounded perturbation of $\re N$ and $\im M-\im N$ is a relatively bounded perturbation of $\im N$. Passing to the real and imaginary parts, we arrive at the situation when we deal with an unperturbed pair $(A_1,A_2)$ and a perturbed pair $(B_1,B_2)$ of commuting self-adjoint operators. 
Note that the most natural way to define the notion of commuting self-adjoint operators is to say that {\it two self-adjoint operators commute if their spectral measures commute}.

Our assumption that $B_1-A_1$ is a relatively bounded perturbation of $A_1$ and $B_2-A_2$ is a relatively bounded perturbation of $A_2$.
It is well known (see e.g. \cite{AP3}) that this assumption is equivalent to the fact that
$(A_1,A_2)$ and $(B_1,B_2)$ are pairs of commuting self-adjoint operators such that
the operators
\bay
\label{otnogr}
(B_1-A_1)(A_1+\ri I)^{-1}\quad\mbox{and}\quad(B_2-A_2)(A_2+\ri I)^{-1}\quad\mbox{are bounded}.
\ey

The main results of the paper will be given in \S\;\ref{glavnoe}. Namely, we obtain in \S\;\ref{glavnoe} estimates of the deviations of functions of pairs of commuting self-adjoint operators under relatively bounded perturbation. Recall that in the case of functions of single self-adjoint operators similar results were obtained recently in \cite{AP3}.

As in earlier papers an important role throughout the paper will be played by double operator integrals.
In \S\;\ref{dvoinya} we give a brief introduction in double operator integrals and Schur multipliers.

Finally, in \S\;\ref{Besovy} we give a brief introduction in the Besov class $B_{\be,1}^1(\R^2)$ which plays an importantrole in perturbation theory.

\

\section{\bf Double operator integrals and Schur multipliers}
\label{dvoinya}
\setcounter{equation}{0}

\

In this section we give a brief introduction to double operator integrals and Schur multipliers. We refer the reader to
\cite{BS1}, \cite{BS2}, \cite{AP1}, \cite{Pe2} and \cite{Pe3} for more detailed information.

Double operator integrals are expressions of the form
\bay
\label{dvopi}
\iint_{\X\times\Y}\Phi(x,y)\,\rd E_1(x)Q\,\rd E_2(y).
\ey
Here $E_1$ and $E_2$ are spectral measures on Hilbert space, $\Phi$ is a bounded measurable function and 
$Q$ is a bounded linear operator on Hilbert space. Such double operator integrals appeared first in the paper \cite{DK} by Yu.L. Daletskii and S.G. Krein. Later M.Sh. Birman  and M.S. Solomyak in \cite{BS1} created a rigorous beautiful theory of double operator integrals. They started with the case when $Q$ is a Hilbert--Schmidt operator and defined the double operator integral in \rf{dvopi} as
$$
\left(\int_{\X\times\Y}\Phi\,\rd{\rm E}\right)Q,
$$
where ${\rm E}$ is the spectral measure on the Hilbert Schmidt class $\bS_2$, which is defined on the measurable rectangles by
$$
{\rm E}(\L\times\D)Q=E_1(\L)QE_2(\D),\quad Q\in\bS_2.
$$
It follows from the definition that in this case the double operator integral in \rf{dvopi} belongs to $\bS_2$ and
$$
\left\|\iint\Phi(x,y)\,\rd E_1(x)Q\,\rd E_2(y)\right\|_{\bS_2}\le\|\Phi\|_{L^\be}\|Q\|_{\bS_2}.
$$

A bounded measurable function $\Phi$ on $\X\times\Y$ is called a {\it Schur multiplier (of the trace class $\bS_1$) with respect to the spectral measures $E_1$ and $E_2$} if
 $$
 Q\in\bS_1\quad\Longrightarrow\quad\iint\Phi\,\rd E_1Q\,\rd E_2\in\bS_1.
 $$
 We denote the class of Schur multipliers of $\bS_1$ with respect to spectral measures $E_1$ and $E_2$ by $\fM_{E_1,E_2}$.
 
 If $\Phi\in\fM_{E_2,E_1}$, we can define by duality the double operator integral in \rf{dvopi} for arbitrary bounded linear operator $Q$ and the transformer
 $$
 Q\mapsto\iint\Phi\,\rd E_1Q\,\rd E_2
 $$
 becomes a bounded linear operator in the operator norm.

In this paper we consider the situation when the spectral measures $E_1$ and $E_2$ are defined on the Borel subsets of the Euclidean space $\R^2$. We denote by $\fM(\R^2\times\R^2)$ the class of {\it Borel Schur multipliers} on $\R^2\times\R^2$,
i.e. the class of of Borel functions on $\R^2\times\R^2$ that are Schur multipliers with respect to arbitrary spectral measures on the Borel subsets of the Euclidean space $\R^2$.

We proceed now to the definition of the Haagerup tensor product of subspaces of $L^\be$ spaces.

 Suppose that $E_1$ and $E_2$ are spectral measures. The {\it Haagerup tensor product} 
 \lb$L^\be_{E_{A_1}}\otimes_{\rm h}L^\be_{E_{A_2}}$ of $L^\be_{E_{A_1}}$ and $L^\be_{E_{A_2}}$ is, by definition, the class of functions $\Phi$
on $\R^2\times\R^2$ of the form
\bay
\label{PhiHaa}
\Phi(x,y)=\sum_{n\ge0}\f_n(x)\psi_n(y),\quad\mbox{for}\quad\f_n\in L^\be_{E_{A_1}}
\quad\mbox{and}\quad\psi_n\in L^\be_{E_{A_2}}
\ey
such that 
\bay
\label{uslHaa}
\sum_{n\ge0}|\f_n|^2\in L^\be_{E_{A_1}}\quad\mbox{and}\quad\sum_{n\ge0}|\psi_n|^2\in L^\be_{E_{A_2}}.
\ey
The norm of $\Phi$ in $L^\be_{E_{A_1}}\otimes_{\rm h}L^\be_{E_{A_2}}$ is the infimum of the expression
$$
\left\|\left(\sum_{n\ge0}|\f_n|^2\right)^{1/2}\right\|_{L^\be}\cdot\left\|\left(\sum_{n\ge0}|\psi_n|^2\right)^{1/2}\right\|_{L^\be}
$$
over all representations of $\Phi$ in the form of \rf{PhiHaa}.

There are various descriptions of the space $\fM_{E_1,E_2}$ of Schur multipliers, see \cite{Pe}, \cite{AP1} and \cite{AP2}. We mention the following one given in \cite{Pe}: $\Phi\in\fM_{E_1,E_2}$ if and only 
$\Phi\in L^\be_{E_1}\otimes_{\rm h}L^\be_{E_2}$. Note that the implication
$$
\Phi\in L^\be_{E_1}\otimes_{\rm h}L^\be_{E_2}\quad\Longrightarrow\quad\Phi\in\fM_{E_1,E_2}
$$
was obtained earlier in \cite{BS2}.

Note also that in \cite{AP2} it was shown that
$$
\|\Phi\|_{\fM_{E_1,E_2}}=\|\Phi\|_{L^\be_{E_1}\otimes_{\rm h}L^\be_{E_2}},\quad\Phi\in\fM_{E_1,E_2}.
$$

\

\section{\bf The Besov class $\bs{B_{\be,1}^1(\R^2)}$}
\label{Besovy}
\setcounter{equation}{0}

\

In this section we briefly introduce the Besov space $B_{\be,1}^1(\R^2)$. More detailed information about Besov spaces 
in more general situations can be found in \cite{N}, \cite{Pee}, \cite{Pe3} and \cite{AP1}.

Let $w$ be an infinitely differentiable function on  $\R$ such that
\bay
\label{w}
w\ge0,\quad\supp w\subset\left[1/2,2\right],\quad\mbox{and} \quad w(t)=1-w\left(\frac t2\right)\quad\mbox{for}\quad t\in[1,2].
\ey

We define the functions $W_n$, $n\in\Z$, on $\R^2$ by 
$$
\big(\F W_n\big)(x)=w\left(\frac{\|x\|}{2^n}\right),\quad n\in\Z, \quad x=(x_1,x_2),
\quad\|x\|\df\Big(x_1^2+x_2^2\Big)^{1/2},
$$
where $\F$ is {\it Fourier transform} defined on $L^1\big(\R^2\big)$ by
$$
\big(\F f\big)(t_1,t_2)=\!\int\limits_{\R^d} f(x)e^{-{\rm i}(x_1t_1+x_2t_2)}\,\rd x_1\,\rd x_2,\!\quad 
$$
Clearly,
$$
\sum_{n\in\Z}(\F W_n)(t)=1,\quad t\in\R^2\setminus\{0\}.
$$

With each tempered distribution $f$ in ${\mathscr S}^\prime\big(\R^2\big)$ we associate the sequence 
$\{f_n\}_{n\in\Z}$,
\bay
\label{fn}
f_n\df f*W_n.
\ey
The formal series $\sum_{n\in\Z}f_n$, being a Littlewood--Payley type expansion of $f$, does not necessarily converges to $f$. 
We start with the definition of the (homogeneous) Besov class $\dot B^1_{\be,1}\big(\R^2\big)$,
$s\in\R$, $0<p,\,q\le\be$, as the space of distributions
$f$ of class ${\mathscr S}^\prime(\R^2)$ (i.e., tempered distributions) such that
\bay
\label{Wn}
\{2^{n}\|f_n\|_{L^\be}\}_{n\in\Z}\in\ell^1(\Z),\quad
\|f\|_{B^1_{\be,1}}\df\big\|\{2^{n}\|f_n\|_{L^\be}\}_{n\in\Z}\big\|_{\ell^1(\Z)}.
\ey
In accordance with this definition, $\dot B^s_{p,q}(\R^d)$ contains all polynomials; next, $\|f\|_{B^1_{\be,1}}=0$ for an arbitrary polynomial $f$. Moreover, the distribution $f$ is determined by the sequence $\{f_n\}_{n\in\Z}$
uniquely up to a polynomial. It is easy to see that the series 
$\sum_{n\ge0}f_n$ converges in\footnote{The space 
${\mathscr S}^\prime(\R^2)$ is equipped with the weak topology
$\s\big({\mathscr S}^\prime(\R^2),{\mathscr S}(\R^2)\big)$.} ${\mathscr S}^\prime(\R^d)$. 
However, the series $\sum_{n<0}f_n$ can diverge in general. Nevertheless, it can be shown that the series
\bay
\label{ryad}
\sum_{n<0}\frac{\partial f_n}{\partial x_1}\quad\mbox{and}\quad\sum_{n<0}\frac{\partial f_n}{\partial x_2}
\ey
converge uniformly in $\R^2$.

Now we can define the modified (homogeneous) Besov space $B^s_{p,q}\big(\R^d\big)$. We say that a distribution $f$
belongs to the Besov class $B^s_{p,q}(\R^d)$  \rf{Wn} is satisfied and
$$
\frac{\partial f}{\partial x_1}=\sum_{n\in\Z}\frac{\partial f_n}{\partial x_1}\quad
\mbox{and}\quad 
\frac{\partial f}{\partial x_2}=\sum_{n\in\Z}\frac{\partial f_n}{\partial x_2}\quad
$$
in the space ${\mathscr S}^\prime\big(\R^2\big)$. Now the function $f$ is uniquely determined by the sequence $\{f_n\}_{n\in\Z}$ up to a constant. Moreover, a polynomial $g$ belongs to 
$B^1_{\be,1}\big(\R^2\big)$ if and only if it is a constant.

\

\section{\bf The main result}
\label{glavnoe}
\setcounter{equation}{0}

\

The key role in this paper will be played by the integral formula given in the following theorem.
We are going to use the notation
$$
x=(x_1,x_2)\in\R^2\quad\mbox{and}\quad y=(y_1,y_2)\in\R^2.
$$
Also, for pairs of commuting self-adjoint operators $(A_1,A_2)$ and $(B_1,B_2)$, we use the notation $E_A$ for the joint spectral measure of the pair $(A_1,A_2)$ and $E_B$ for the joint spectral measure of the pair $(B_1,B_2)$.

\begin{thm}
\label{osnfor}
Suppose that $(A_1,A_2)$ and $(B_1,B_2)$ are pairs of commuting self-adjoint operators satisfying {\em\rf{otnogr}}.
Suppose that $f$ is a continuous function on $\R^2$ such that the functions
$$
(x,y)\mapsto\frac{f(x_1,y_2)-f(x_1,x_2)}{y_2-x_2}(x_2+\ri),\quad(x,y)\in\R^2\times\R^2,
$$
and
$$
(x,y)\mapsto\frac{f(y_1,y_2)-f(x_1,y_2)}{y_1-x_1}(x_1+\ri),\quad(x,y)\in\R^2\times\R^2,
$$
belong to the space $\fM(\R^2\times\R^2)$ of Borel Schur multipliers.
\begin{align}
\label{intpre}
f&(B_1,B_2)-f(A_1,A_2)\nonumber\\[.2cm]
&=
\iint_{\R^2}\frac{f(x_1,y_2)-f(x_1,x_2)}{y_2-x_2}(x_2+\ri)
\,\rd E_B(y_1,y_2)(B_2-A_2)(A_2+\ri I)^{-1}\,\rd E_A(x_1,x_2)\nonumber\\[.2cm]
&+
\iint_{\R^2}\frac{f(y_1,y_2)-f(x_1,y_2)}{y_1-x_1}(x_1+\ri)
\,\rd E_B(y_1,y_2)(B_1-A_1)(A_1+\ri I)^{-1}\,\rd E_A(x_1,x_2)
\end{align}
\end{thm}

%
%
%
%
%
%

First, we establish several auxiliary facts.
For a function $f$ on $\R^2$, we put
$$
f_t(s)\df f(t,s)\quad\mbox{and}\quad f^s(t)\df f(t,s).
$$

\begin{lem}
\label{triv}
Let $f$ be a function on $\R^2$.
Suppose that
\bay
\label{triv01}
|(f(x_1,y_2)-f(x_1,x_2))(x_2+\ri)|\le C|y_2-x_2|
\ey
and
\bay
\label{triv02}
|f(y_1,y_2)-f(x_1,y_2))(x_1+\ri)|\le C|y_1-x_1|
\ey
for all $x,y\in\R^2$.
Then $f$ is bounded.
\end{lem}

\Pf Condition \rf{triv01} can be rewritten as follows
$$
|f_t(s_1)-f_t(s_2)|\le C\frac{|s_1-s_2|}{\sqrt{s_1^2+1}}
$$
for all $t\in\R$ and $s\in\R^2$. Substituting $s_2=0$ we get  $|f(t,s)-f(t,0)|\le C$
for all $s,t\in\R$.

In the same way we have $|f(t,s)-f(0,s)|\le C$ for all $s,t\in\R$.
It remains to note that
$$
|f(t,s)-f(0,0)|\le|f(t,s)-f(t,0)|+|f(t,0)-f(0,0)|\le 2C
$$
and $|f(t,s)|\le 2C+|f(0,0)|$ for all $t,s\in\R$. $\bl$

\medskip

We also need the following assertion.

\begin{lem}
\label{triv1}
Let $f$ be a function on $\R^2$. Suppose that
\bay
\label{triv11}
|f(x_1,y_2)(y_2+\ri)-f(x_1,x_2)(x_2+\ri)|\le C|y_2-x_2|
\ey
and
\bay
\label{triv12}
|f(y_1,y_2)(y_1+\ri)-f(x_1,y_2)(x_1+\ri)|\le C|y_1-x_1|
\ey
for all $x,y\in\R^2$
Then $f$ is bounded.
\end{lem}

\Pf Condition \rf{triv11} can be rewritten as follows
$$
|f_t(s_1)(s_1+\ri)-f_t(s_2)(s_2+\ri )|\le C|s_1-s_2|
$$
for all $t\in\R$ and $s\in\R^2$. Substituting $s_2=t=0$ we get \\
$|f_0(s)(s+\ri)-f(0,0)\ri |\le C|s|$
for all $s\in\R$. Hence, $f_0$ is a bonded function. Let $|f_0|\le C_1$.

Condition  \rf{triv12} can be rewritten as follows
$$
|f^s(t_1)(t_1+\ri)-f^s(t_2)(t_2+\ri)|\le C|t_1-t_2|
$$
for all $s\in\R$ and $t\in\R^2$. Hence,
$$
|f(t,s)(t+\ri)-f(0,s)\ri |\le C|t|
$$
for all $s,t\in\R$.

It remains to observe that
$$
|f(t,s)|\le\frac{C|t|+C_1}{|t+\ri|}.\quad\bl
$$

\medskip

{\bf Proof of Theorem \ref{osnfor}.} Assume first that the self-adjoint operators in question are bounded. Obviously, the divided differences 
$$
(x,y)\mapsto\frac{f(x_1,y_2)-f(x_1,x_2)}{y_2-x_2}\quad\mbox{and}\quad(x,y)\mapsto\quad\frac{f(y_1,y_2)-f(x_1,y_2)}{y_1-x_1}
$$
are Schur multipliers with respect to $E_A$ and $E_B$.

We have 
\begin{multline*}
\iint_{\R^2}\frac{f(x_1,y_2)-f(x_1,x_2)}{y_2-x_2}(x_2+\ri)
\,\rd E_B(y_1,y_2)(B_2-A_2)(A_2+\ri I)^{-1}\,\rd E_A(x_1,x_2)\\[.2cm]
=\iint_{\R^2}\frac{f(x_1,y_2)-f(x_1,x_2)}{y_2-x_2}
\,\rd E_B(y_1,y_2)(B_2-A_2)\,\rd E_A(x_1,x_2)
\end{multline*}
Clearly, the function 
$$
(x,y)\mapsto f(x_1,y_2)-f(x_1,x_2)
$$
must also be a Schur multipliers with respect to $E_A$ and $E_B$.

By Lemma \ref{triv}, the function $f$ must be bounded, and so the function
$(x,y)\mapsto f(x_1,x_2)$, being independent of $y$, must be in $\fM(\R^2\times\R^2)$
which in turn implies that the function $(x,y)\mapsto f(x_1,y_2)$ must be a Schur multipliers with respect to $E_A$ and $E_B$.

Thus,
\begin{multline*}
\iint_{\R^2}\frac{f(x_1,y_2)-f(x_1,x_2)}{y_2-x_2}
\,\rd E_B(y_1,y_2)(B_2-A_2)\,\rd E_A(x_1,x_2)\\[.2cm]
=\iint_{\R^2}(f(x_1,y_2)-f(x_1,x_2))
\,\rd E_B(y_1,y_2)\,\rd E_A(x_1,x_2)\\[.2cm]
=\iint_{\R^2}f(x_1,y_2)\,\rd E_B(y_1,y_2)\,\rd E_A(x_1,x_2)-f(A_1,A_2).
\end{multline*}

Similarly,
\begin{multline*}
\iint_{\R^2}\frac{f(y_1,y_2)-f(x_1,y_2)}{y_1-x_1}(x_1+\ri)
\,\rd E_B(y_1,y_2)(B_1-A_1)(A_1+\ri I)^{-1}\,\rd E_A(x_1,x_2)\\[.2cm]
=f(B_1,B_2)-\iint_{\R^2}f(x_1,y_2)\,\rd E_B(y_1,y_2)\,\rd E_A(x_1,x_2).
\end{multline*}

This proves formula \rf{intpre} in the case of bounded operators. The reduction to the case of possibly unbounded operators is more or less standard, see, e.g. the proof of Theorem 5.2 in \cite{APPS}.

Indeed, clearly, the double operator integrals on the right-hand side of \rf{intpre} determine bounded linear operators. Multiply them by the spectral projections 
$$
P_M\df E_B(\{(y_1,y_2):~|y_1|\le M~\;\mbox{and}~\;|y_2|\le M\}
$$
on the left and by the spectral projections
$$
Q_M\df E_B(\{(x_1,x_2):~|x_1|\le M~\;\mbox{and}~\;|x_2|\le M\}
$$
on the right. Obviously, the resulting sequence converges to the right-hand side of \rf{intpre} in the strong operator topology. It remains to observe that 
formula \rf{intpre} holds if we replace the pairs $(A_1,A_2)$ and $(B_1,B_2)$ with the pairs
$(Q_MA_1,Q_MA_2)$ and $(P_MB_1,P_MB_2)$ of bounded commuting self-adjoint operators
and proceed to the limit as $M\to\be$ in the strong operator tipology. $\bl$

\begin{lem}
\label{umnnas+i}
Let $f$ be a bounded function on $\R^2$ such that the function $g$ defined by
$$
g(s,t)=f(s,t)(s+i)
$$
belongs to the Besov class $B_{\be,1}^1(\R^2)$. Then the function
$$
(x,y)\mapsto\frac{f(y_1,y_2)(y_1+\ri)-f(x_1,y_2)(x_1+\ri)}{y_1-x_1}
$$
on $\R^2\times\R^2$ belongs to $\fM(\R^2\times\R^2)$,
where
$$
x=(x_1,x_2)\quad\mbox{and}\quad y=(y_1,y_2).
$$
\end{lem}

\Pf The result can easily be obtained from Theorem 6.2 of \cite{APPS}. $\bl$

\begin{cor}
Under the hypotheses of Lemma {\em\ref{umnnas+i}} the function
$$
(x,y)\mapsto\frac{f(y_1,y_2)-f(x_1,y_2)}{y_1-x_1}(x_1+\ri)
$$
on $\R^2\times\R^2$ belongs to $\fM(\R^2\times\R^2)$.
\end{cor}

\Pf We have
$$
\frac{f(y_1,y_2)-f(x_1,y_2)}{y_1-x_1}(x_1+\ri)=\frac{f(y_1,y_2)(y_1+\ri)-f(x_1,y_2)(x_1+\ri)}{y_1-x_1}-f(y_1,y_2).
$$
Since the function 
$$
(x,y)\mapsto f(y_1,y_2)
$$
is bounded and depends only on $y$, it belongs to $\fM(\R^2\times\R^2)$. The result follows now from
Lemma \ref{umnnas+i}. $\bl$

\begin{cor}
Under the hypotheses of Lemma {\em\ref{umnnas+i}}
$$
\|f(B_1,B_2)-f(A_1,B_2)\|\le\const\|g\|_{B_{\be,1}^1}
\Big\|(B_1-A_1)(A_1+\ri I)^{-1}\Big\|.
$$
\end{cor}

\begin{thm}
\label{vterBes}
Let $(A_1,A_2)$ and $(B_1,B_2)$ be pairs of commuting self-adjoint operators satisfying
{\em\rf{otnogr}}.
Suppose that $f$ is a function on $\R^2$ such that the functions $g_1$ and $g_2$ defined by
$$
g_1(s,t)=f(s,t)(s+i)\quad\mbox{and}\quad g_2(s,t)=f(s,t)(t+i)
$$
belong to the Besov class $B_{\be,1}^1(\R^2)$. Then the functions
$$
(x,y)\mapsto\frac{f(y_1,y_2)-f(x_1,y_2)}{y_1-x_1}(x_1+\ri)
$$
and
$$
(x,y)\mapsto\frac{f(x_1,y_2)-f(x_1,x_2)}{y_2-x_2}(x_2+\ri)
$$
are Borel Schur multipliers on $\R^2\times\R^2$ and
\begin{multline*}
\big\|f(B_1,B_2)-f(A_1,A_2)\big\|\le\const\big(\|g_1\|_{B_{\be,1}^1}+\|g_2\|_{B_{\be,1}^1}\big)\\
\times\max\Big\{\big\|(B_1-A_1)(A_1+\ri I)^{-1}\big\|,\big\|(B_2-A_2)(A_2+\ri I)^{-1}\big\|\Big\}.
\end{multline*}
\end{thm}

Suppose now that $1\le p<\be$ and that $(A_1,A_2)$ and $(B_1,B_2)$ are pairs of commuting  self-adjoint operators such that 
$(B_j-A_j)$ is a relatively $\bS_p$ perturbation of $A_j$, $j=1,2$, i.e.
\bay
\label{otnSp}
(B_1-A_1)(A_1+\ri I)^{-1}\in\bS_p\quad\mbox{and}\quad(B_2-A_2)(A_2+\ri I)^{-1}\in\bS_p,
\ey
where $\bS_p$ denotes the Schatten--von Neumann class.

\begin{thm}
\label{otnSpvoz}
Let $1\le p<\be$ and let $(A_1,A_2)$ and $(B_1,B_2)$ be pairs of commuting self-adjoint operators satisfying {\em{\rf{otnSp}}}.
Suppose that $f$ is a function on $\R^2$ satisfying the hypotheses of Theorem {\em\ref{vterBes}}. Then
$f(B_1,B_2)-f(A_1,A_2)\in\bS_p$ and there exists a positice number $K$ such that
$$
\|f(B_1,B_2)-f(A_1,A_2)\|_{\bS_p}\le K\max\big\{\|B_1-A_1\|_{\bS_p},\|B_2-A_2\|_{\bS_p}\big\}.
$$
\end{thm}

The theorem can be easily deduced from Theorems \ref{osnfor} and \ref{vterBes} with the help of interpolation.

\

\

\begin{footnotesize}

\
 
\noindent
\begin{tabular}{p{8cm}p{15cm}}
A.B. Aleksandrov & V.V. Peller \\
St.Petersburg State University & St.Petersburg State University \\
Universitetskaya nab., 7/9  & Universitetskaya nab., 7/9\\
199034 St.Petersburg, Russia & 199034 St.Petersburg, Russia \\
\\

St.Petersburg Department &St.Petersburg Department\\
Steklov Institute of Mathematics  &Steklov Institute of Mathematics  \\
Russian Academy of Sciences  & Russian Academy of Sciences \\
Fontanka 27, 191023 St.Petersburg &Fontanka 27, 191023 St.Petersburg\\
Russia&Russia\\
email: alex@pdmi.ras.ru&email: pellerv@gmail.com
\end{tabular}
\end{footnotesize}

\end{document}